\documentclass{article}

\usepackage{amsthm}
\usepackage{amsfonts}
\usepackage{amssymb}
\usepackage{verbatim}
\usepackage[all]{xy}
\usepackage{amsmath, amscd, amsthm}
\usepackage{longtable}
\usepackage{amscd}
\usepackage{mathrsfs}
\usepackage{graphicx}
\usepackage{latexsym}

\newtheorem{theorem}{Theorem}[section]
\newtheorem{lemma}[theorem]{Lemma}
\newtheorem{prop}[theorem]{Proposition}
\newtheorem{cor}[theorem]{Corollary}
\newtheorem{ex}[theorem]{Example}
\newtheorem{defn}[theorem]{Definition}
\newtheorem{rem}[theorem]{Remark}
\newtheorem{question}[theorem]{Question}

\usepackage{latexsym}

\newcommand{\z}{{\mathbb Z}}

\begin{document}

\title{Relative Subgroup Growth and Subgroup Distortion}
\author{Tara C. Davis, Alexander Yu. Olshanskii}
\maketitle

\abstract{We study the relative growth of finitely generated subgroups in finitely generated groups, and the corresponding distortion function of the embeddings. We explore which functions are equivalent to the relative growth functions and distortion functions of finitely generted subgroups. We also study the connections between these two asymptotic invariants of group embeddings. We give conditions under which a length function on a finitely generated group can be extended to a length function on a larger group.}

\section{Introduction and Background}

We will use the notation that for a group $G$ with finite generating set $X$, and for an element $g \in G$, then $|g|_X$ represents the word length of the element $g$ with respect to the generating set $X$. For this paper, $\mathbb{N} = \{1,2, \dots,\}$. For $r \in \mathbb{N}$ we use the notation that $B_G(r)$ represents the ball in $G$ centered at the identity with radius $r$. That is, $B_G(r) = \{ g \in G : |g|_X \leq r\}$.

\begin{defn}
Let $G$ be a group finitely generated by a set $X$ with any subset $H$. The growth function of $H$ in $G$ is $$g_H^G: \mathbb{N} \rightarrow \mathbb{N}: r \rightarrow \# \{ h \in H: |h|_X \leq r\}=\#(B_G(r) \cap H).$$
\end{defn}

The growth functions of subsets and other functions of this type become asymptotic invariants independent of the choice of a finite set $X$ if one takes them up to a certain equivalence $\sim$ (see the definitions of $\sim$, $\preceq$, $\succeq$ and $\approx$ in Section \ref{three}.) 

The growth of a subgroup $H$ with respect to a finite set of generators of a bigger group $G$ is called the {\it relative growth} of $H$ in $G$ and denoted by $g_H^G$. In particular, $g_G(r)= g_G^G$ is the usual growth function of the entire finitely generated group $G$.

The relative growth of subgroups, as an asymptotic invariant independent on the choice of generating sets, was studied by Osin in \cite{osin}. He provided a description of relative growth functions of cyclic subgroups in solvable groups, up to a rougher (so called polynomial, based on inequalities of the form $f(r)\le \phi_1(g(\phi_2(r)))$ for some polynomials $\phi_1, \phi_2$ depending on the functions $f$ and $g$) equivalence relation than the ones we employ, $\sim$ and $\approx$. He also provided a negative example to the following question attributed to A.Lubotzky by Grigorchuk and De La Harpe (\cite{gdlh}): Is it true that the relative growth functions of subgroups in solvable groups (linear groups) are either polynomial or exponential? However, the relative growth of any finitely generated subgroup of a free solvable groups is either exponential or polynomial. Moreover, it was shown in \cite{osins} that the relative growth of any subgroup of a policyclic group is either exponential or at most polynomial.

\begin{defn}\label{distortion}(\it{cf.} \cite{gromov2})
Let $H$ be a subgroup of a group $G$, where $H$ and $G$ are generated by finite sets $Y$ and $X$, respectively. Then the distortion function of $H$ in $G$ is defined as $$\Delta^{G}_{H} : \mathbb{N} \rightarrow \mathbb{N} : r \mapsto \max \{ |h|_Y : h \in H, |h|_X \leq r \}.$$ 
\end{defn}

The distortion function does not depend on the choice of finite generating sets $X$ and $Y$, so long as it is considered only up to the equivalence of Definition \ref{equivalence}. Moreover, if the subgroup $H$ is infinite, then the distortion function of $H$ in $G$ is at least linear up to equivalence.

\begin{defn}
The finitely generated subgroup $H$ of the finitely generated group $G$ is said to be undistorted if $\Delta^{G}_{H}(r) \approx r$.
\end{defn}

If a subgroup $H$ is not undistorted, then it is said to be distorted, and its distortion refers to the equivalence class of $\Delta_H^G(r)$. 

\begin{ex}\label{pok} (\it{cf.} \cite{gromov2})
\begin{enumerate}
\item Consider the three-dimensional Heisenberg group $$\mathcal{H}^3 = \langle a, b, c | c=[a,b], [a,c]=[b,c] = 1 \rangle.$$ Consider the cyclic subgroup $H = \langle c \rangle_{\infty}$. Then $H$ is distorted and has quadratic distortion.

\item Consider the Baumslag-Solitar Group BS$(1,2) = \langle a, b | b a b^{-1} = a^2 \rangle$. 
It has infinite cyclic subgroup $\langle a \rangle$ with exponential distortion. 

\end{enumerate}
\end{ex}

\section{Summary of Main Results}
In Section \ref{ekh} we prove the following result which classifies which functions are equivalent to relative growth functions of infinite cyclic subgroups in finitely generated groups.

Recall that for a function $g: \mathbb{N} \rightarrow \mathbb{N}$ we say that $g$ is superadditive if $g(r+s) \geq g(r) + g(s),$ and that $g$ is subadditive if $g(r+s) \leq g(r) + g(s)$, for all $r,s$.

\begin{theorem}\label{2.2}
\begin{enumerate}
\item If $H$ is a non-locally-finite subgroup of a finitely generated group $G$, then the function $g_H^G$ is $\sim$-equivalent to a superadditive function $f: \mathbb{N} \rightarrow \mathbb{N}$.
\item If $H$ is a countably infinite and locally finite group, then there is an embedding of $H$ into a finitely generated group $G$ so that $g_H^G$ is not equivalent to any nonzero superadditive  function.
\item Let $H$ be infinite cyclic, and let $f: \mathbb{N} \rightarrow \mathbb{N}$ be any non-zero superadditive  function of at most exponential growth; that is, $f(r) \preceq 2^r$. Then there exists a finitely generated (solvable) group $G$ and an embedding of $H$ into $G$ such that the relative growth function of the embedding is equivalent to the function $f$.
\end{enumerate}
\end{theorem}

\begin{rem}
One can observe a discrepancy between the conditions
in the above Theorem \ref{2.2} and in Theorem 1 about polynomial
equivalence from \cite{osin}. Moreover, there are 
functions equivalent to  superadditive ones, which do not
satisfy Condition (a) of Theorem 1 in \cite{osin}. The correction
is easy: In \cite{osin}, one should not require that Condition (a)
holds for the function, but to say that  the relative growth
function is $\sim \mathrm{equivalent}$ to a function satisfying 
Condition (a). This does not affect the other results from \cite{osin}.
\end{rem}

In Section \ref{pdicg} we investigate the following. The paper \cite{olshanskii} gives information about how to obtain various distortions of infinite cyclic subgroups. The authors were asked by Mark Sapir about what all possible distortions for the embeddings of infinite cyclic subgroups are.

\begin{theorem}\label{s12}
\begin{enumerate}
\item Let $H$ be an infinite, finitely generated (solvable) group, and let $f$ be an increasing and superadditive function on $\mathbb{N}$. Then there exists an embedding of $H$ into a finitely generated (solvable) group $G$ so that $\Delta_H^G \approx f$.
\item Any distortion function of an infinite cyclic group is equivalent to a superadditive and increasing function.
\item There is an embedding of a finitely generated (solvable) group $H$ into a finitely generated (solvable) group $G$ so that the distortion function is not equivalent to a superadditive function. Moreover, the distortion is not bounded by a recursive function, but it is bounded by a linear function on an infinite subset of $\mathbb{N}$. 
\item There is a \emph{finitely presented} group $K$ with a distorted finitely generated subgroup $H$ such that the distortion of $H$ is bounded by a linear function on an infinite subset of $\mathbb{N}$ (and so nonequivalent to a superadditive function).
\end{enumerate}
\end{theorem}

The statements $(1)$ and $(2)$ of Theorem \ref{s12} yield the answer to Sapir's question:

\begin{cor}
If $H$ is an infinte cyclic group, then all possible distortion functions for the embeddings of H into finitely generated (solvable) groups G are, up to the equivalence $\approx$ all the increasing superadditive functions $f: \mathbb{N} \rightarrow \mathbb{N}$.
\end{cor}

It is interesting to compare this class of functions with another class of functions. In \cite{osinnilp}, Corollary 2.4, Osin obtained the description of distortion functions for subgroups in finitely generated nilpotent groups: they can be only of the form $r^q$ for some rational exponents $q$. 

We are able to make the following connections with the relative growth functions of embeddings.

\begin{theorem}\label{ellp}
\begin{enumerate}
\item Let $f$ be an increasing, superadditive  function on $\mathbb{N}$. The embedding of $H=\z$ into a finitely generated group $G$ having distortion function equivalent to $f$ given by Theorem \ref{s12} can be chosen so that the relative growth of $H$ in $G$ is bounded up to equivalence $\approx$ from above by $2^{\sqrt{r}}$. 
\item The function $2^{\sqrt{r}}$ in $(1)$ cannot be replaced by a function with less growth, since if the distortion function of a finitely generated subgroup is exponential, then its relative growth is at least $2^{\sqrt{r}}$, up to equivalence.
\end{enumerate}
\end{theorem}

One of the motivations for studying recursive bounds on the distortion function is because of connections with the computational complexity of the algorithmic membership problem. It was observed in \cite{gromov} and proved in \cite{farb} that for a finitely generated subgroup $H$ of a finitely generated group $G$ with solvable word problem, the membership problem is solvable in $H$ if and only if the distortion function of $H$ in $G$, $\Delta_{H}^{G}(r)$, is bounded by a recursive function. 

The next result will be proved in Section \ref{swz}. The authors will recall the definition of the little $o$ notation, as well as explain the meaning of the effective version of little $o$, denoted by $o^e$, in Section \ref{three}. Although the statement \ref{cm} Part $1$ is well known, we include it in the statement of the Theorem for completeness.

\begin{theorem}\label{cm}
\begin{enumerate}
\item  If $H$ be a finitely generated subgroup of a finitely generated group $G$ and $g_H^G(r)$ is $o(r^2)$, then $H$ must be virtually cyclic.
\item Let $H$ be a finitely generated subgroup of a finitely generated group $G$. If $g_H^G(r)$ is $o^e(r^2)$, then $\Delta_H^G(r)$ is bounded above by a recursive function.

\item One may not replace $o^e(r^2)$ by $o(r^2)$ in the previous statement.
\end{enumerate}
\end{theorem}

\begin{rem}
Of course $o(r^2) \preceq 2^{\sqrt{r}}$; but Part $(1)$ of Thereom \ref{ellp} does not contradict Part $(3)$ of Theorem \ref{cm} because there are functions $\mathbb{N} \rightarrow \mathbb{N}$ which are not bounded by any recursive function from above and not bounded from below by any exponential function.
\end{rem}

In Section \ref{three} we introduce another equivalence {\it of functions on groups}, $\Theta$ equivalence.

We recall the following result of Olshanskii (see \cite{olshanskii}).

\begin{theorem}\label{olshd}
Let $H$ be a group and $l: H \rightarrow \{0,1,2,\dots,\}$ be a function satisfying: 
\begin{itemize}
\item (D1) $l(h)=l(h^{-1}), h \in H; l(h)=0$ if and only if $h=1$
\item (D2) $l(h_1h_2) \leq l(h_1)+l(h_2), h_1, h_2 \in H$
\item (D3) There exists $a>0$ such that $\#\{h \in H : l(h) \leq r\} \leq a^r$ for any $r \in \mathbb{N}.$
\end{itemize}
Then there exists a finitely generated group $G$ with generating set $X$ and an embedding of $H$ into $G$ such that the function $h \mapsto |h|_X$, $h \in H$ is $\Theta$-equivalent to the function $l$.
\end{theorem}

We refer to conditions $(D1), (D2),$ and $(D3)$ as the $(D)$ condition, and we often denote functions $l: H \rightarrow \mathbb{N} $ satisfying the $(D)$ condition as length functions. 

In Section \ref{elf} we study conditions under which a length function $l : H \rightarrow \{0,1,2,\dots,\} $ on a finitely generated group $H$ can be extended to a larger finitely generated group $G$ containing $H$ as a subgroup. This is a natural thing to study, because the proofs of the earlier Theorems were based on length functions, so we want to offer in conclusion some results on length functions themselves.

\begin{defn}
Let $H$ be a finitely generated subgroup of a finitely generated group $G$ and let $l : H \rightarrow \{0,1,2,\dots,\}$ satisfy the $(D)$ condition. We say that a function $L : G \rightarrow \mathbb{N} $ is an extension of $l$ if
\begin{itemize}
\item $L$ is a length function; that is, it satisfies the $(D)$ condition.
\item $L$ is an extension of $l$ up to $\Theta$-equivalence: the function $L$, when restricted to $H$, is $\Theta$-equivalent to $l$ on $H$.
\end{itemize}
\end{defn}

\begin{defn}\label{special}
Let $H$ be generated by a finite set $Y$ and consider a length function $l$ on $H$. If there exists a subadditive nondecreasing function $f$ such that for any $h \in H$, $l(h)= \Theta (f (|h|_Y))$ then we call the function $l$ special.

\end{defn}

\begin{theorem}\label{kam}
\begin{enumerate}
\item Let $H$ be a finitely generated undistorted subgroup of a finitel generated group $G$, and suppose that the length function $l$ on $H$ is special. Then there exists an extension $L$ of $l$ to $G$.
\item

Let $H$ be a finitely generated group, and let $l : H \rightarrow \{0,1,2,\dots\} $ be any function satisfying the $(D)$ condition. If $l$ is not special in the sense of Definition \ref{special} then there exists an undistorted embedding of $H$ into a finitely generated group $G$ such that $l$  does not extend to $G$.
\item If $H$ is a distorted subgroup of $G$, then there is a special length function $l$ on $H$ admitting no extension to $G$.
\end{enumerate}
\end{theorem}

\section{Embeddings and Relative Growth}\label{three}

In the definition of growth function, one should define an equivalence to get rid of the dependence on the choice of finite generating set $X$. In case of growth functions we need the  following equivalence, $\sim$.
 
For two functions $f,g: \mathbb{N} \rightarrow \mathbb{N}$, say that $f$ does not exceed $g$ up to equivalence if there exists a constant $c$ so that for all $r \in \mathbb{N}$ we have $f(r) \leq g(cr)$, and that $f$ is equivalent to $g$ ($f \sim g$) if both $f$ does not exceed $g$ and $g$ does not exceed $f$, up to equivalence. If we define the growth function to be the equivalence class of $g_H^G$ above, then it becomes independent of the choice of finite generating set. This is because if $X$ and $X'$ are two finite generating sets for $G$, then for the constant $c=\max\{|s|_{X'}:s \in X\}$ we have that $\{g \in G: |g|_X \leq r\} \subseteq \{g \in G: |g|_{X'} \leq cr\}$. 

For distortion functions of subgroups (see Definition \ref{distortion}) one needs another equivalence $\approx$.
\begin{defn}\label{equivalence}
One says that for $f,g: \mathbb{N} \rightarrow \mathbb{N}$ that $f \preceq g$ if there exists $c>0$ so that for all $r$, $f(r) \leq cg(cr)$. Two such functions $f$ and $g$ are said to be equivalent if $f \preceq g$ and $g \preceq f$, and we write this as $f \approx g$.
\end{defn}

The relative growth of an infinite finitely generated subgroup is always superadditive  (up to $\sim$-equivalence, see Theoerm \ref{2.2}) and so its $\sim$ class coincides with the $\approx$ class. This explains the preference for use of $\approx$ in our paper.

We will also need another equivalence of functions on {\it groups}.

\begin{defn}\label{theta}
Let $G$ be a group, and let $f, g : G \rightarrow \mathbb{N}$. We say that $f$ and $g$ are $\Theta$-equivalent, and write $f=\Theta(g)$ if there exists $c>0$ so that for all $r \in G$, $f(r) \leq cg(r)$ and $g(r) \leq cf(r)$.
\end{defn}

We will use some remarkable results regarding the usual growth function. A finitely generated group $G$ with finite generating set $X$ is said to have polynomial growth if $g_G(r) \preceq r^d$ for some $d \in \mathbb{N}$. A group $G$ is said to have exponential growth if for some $d >1$, $g_G(r) \succeq d^r$. It was proved by Wolf in \cite{wolf} that if $G$ is a finitely generated nilpotent group, then $G$ has polynomial growth. The degree of polynomial growth (up to $\sim$) in nilpotent groups is computed by Bass in \cite{bass} and is given by the following explicit formula
\begin{equation}\label{bass}
d=\displaystyle\sum_{k \geq 1}k \cdot \textrm{rk}(G_k/G_{k+1}),
\end{equation} 
where $\textrm{rk}$ represents the torsion free rank of an abelian group, and $G_k$ the terms of the lower central series for $G$. The famous theorem of Gromov says that a finitely generated group $G$ has polynomial growth only if it is virtually nilpotent (see \cite{gromov}). However, there are examples of groups of intermediate growth; that is, groups whose growth function is neither polynomial nor exponential (see \cite{grigorchuk2}).

In \cite{oo}, Corollary 1.2, the following result was obtained.

\begin{theorem}\label{olshsolv}
In the notation of Theorem \ref{olshd}, if $H$ is solvable, then so is $G$. In particular, if $H$ is solvable with solvability length $l$, then the solvability length of $G$ is $l+4$.
\end{theorem}

In Theorem \ref{olshc}, we establish notation for the often occurring special case of Theorems \ref{olshd} and \ref{olshsolv}.

\begin{theorem}\label{olshc}
Let $l: \z \rightarrow \{0,1,2,\dots,\} $ be a function satisfying: for $m,n \in \mathbb{Z}$,
\begin{itemize}
\item (C1) $l(n)=l(-n), n \in \z; l(n)=0$ if and only if $n=0$
\item (C2) $l(n+m) \leq l(n)+l(m)$
\item (C3) There exists $a>0$ such that $\#\{i \in \z : l(i) \leq r\} \leq a^r$ for any $r \in \mathbb{N}.$
\end{itemize}
Then there exists a finitely generated  (solvable) group $G$ with generating set $X$ and an element $g \in G$ such that $$|g^n|_X  =\Theta( l(n)).$$
\end{theorem}

We refer to conditions $(C1), (C2),$ and $(C3)$ as the $(C)$ condition.

\begin{rem}
We may translate the geometric group theoretic functions into different terms as follows. Suppose that $l: H \rightarrow \mathbb{N} $ satisfies the $(D)$ condition, so that we have an embedding $H \rightarrow G$ with all notation as in Theorem \ref{olshd}. Then the relative growth of the finitely generated subgroup $H$ in $G$ is given by $$g_{H}^G(r) \approx \#\{h \in H :l(h) \leq r\},$$ and the distortion is $$\Delta_{H}^G(r) \approx \max \{|h|_H : h \in H, l(h) \leq r\}.$$

\end{rem}

The relative growth function is an asymptotic invariant of the embedding $H \leq G$ into a finitely generated group $G$ and can be studied in contrast with the usual growth functions of $H$ (if $H$ is also fintiely generated) and $G$ defined respectively as: $g_H(r)=\#B_H(r)$ and $g_G(r)=\#B_G(r)$. It is clear that when $H=G$ that $g_H^G(r) = g_G(r)$. Also, if $K \leq H \leq G$ then there are two relative growth functions, $g_K^G(r)$, the relative growth of $K$ in $G$, and $g_H^G(r)$, the relative growth of $H$ in $G$. In this case we have that $g_K^G(r) \preceq g_H^G(r)$.

Observe further that if $H$ is a finitely generated subgroup of a finitely generated group $G$, then we have that 
\begin{equation}\label{ng}
g_H(r) \preceq g_H^G(r) \preceq g_G(r) \preceq 2^r.
\end{equation}
This follows because $B_H(r) \subseteq B_G(cr) \cap H$ where $c$ is a constant depending only on the choice of finite generating sets.Because of the fact that $g_G(r) \leq (2\#X+1)^r$, we see that the regular and relative growth functions are always at most exponential.

We recall a couple more elementary notions. We say that a function $f(r)$ is $o(g(r))$ if $\displaystyle\lim_{r \to \infty}\frac{f(r)}{g(r)}=0$. The {\it effective} limit of a function $g(r)$ is infinity if there is an algorithm that given an integer $C$ computes $N=N(C)$ such that $g(r)\ge C$ for every $r > N$. In fact, the effective limit of a function can be any real number or $\pm \infty$; the effective limit of $f(r)$ equals $M$ if there is an algorithm that, given an integer $n >0$, computes $N=N(n)$ so that $|f(r)-M|<\frac{1}{n}$ for every $r>N$. We introduce the following notation: we say that $$f(r)=o^e(g(r))$$ if the effective limit of $\frac{f(r)}{g(r)}$ is $0$. For example, $(\log \log \log r)^{-1} = o^e(1)$.

The following facts are well-known and easily verified, and will be used implicitly throughout the text.

\begin{lemma}\label{findex}
Let $G$ be a finitely generated group, with finitely generated subgroup $H$.
\begin{itemize}
\item If $K \leq H \leq G$ and $[H:K] < \infty$ then $\Delta_H^G(r) \approx \Delta_{K}^G(r).$
\item If $[G:H] < \infty$ then $g_G(r) \approx g_H(r)$.
\item If $\Delta_H^G(r) \approx r$ then $g_H^G(r) \approx g_H(r)$.
\end{itemize}
\end{lemma}

Therefore, we see that the {\it relative} growth effect appears only for distorted finitely generated subgroups. The study of the relative growth function is already motivated by the fact that it is an asymptotic invariant. However, in light of the above fact: that for undistorted subgroups the relative growth is $\sim$-equivalent to the usual growth, we are naturally curious about the relationship between these two functions. This serves to further justify the questions we address regarding the connections between the two asymptotic invariants.

Let the infinite cyclic group $\z=\langle a \rangle_{\infty}$ be embedded in a finitely generated group $G$ and $\Delta_{a}^G$ be the corresponding distortion function. 

\begin{lemma}\label{hfd}
The function $\Delta_a^G$ is equivalent to an increasing superadditive function.
\end{lemma}

\begin{proof}

Let $a$ be included in the finite generating set of $G$. Let $\Delta_a^G (r) = N$ and $\Delta_a^G (s) = M$. Then the ball $B_G(r)$ of radius $r$ in $G$ has the element $a^{N}$ of $\z$. Likewise $a^M \in B_G(s)$, and so $a^{M +N} \in B_G(r+s)$. It follows that $\Delta_a^G (r+s) \geq M + N = \Delta_a^G (r)+\Delta_a^G (s)$. Therefore the distortion function is superadditive. Moreover, $\Delta(1) > 0$, since the ball of radius $1$ in $G$ contains $a$. By superadditivity, $\Delta(r+1) \ge \Delta(r)+\Delta(1)>\Delta(r)$; i.e. the distortion function strictly increases.

\end{proof}

Now that we have explained all relevant terminology and notation relating to embeddings and relative growth, we would like to introduce some of the elementary connections between the relative growth of a finitely generated subgroup in a finitely generated group, and the correspondng distortion function of the embedding. 

\begin{lemma}\label{erlr}
Suppose that $K$ is an infinite cyclic subgroup of a finitely generated group $G$. If $K$ is distorted, then the relative growth function of $K$ is not equivalent to a linear function.
\end{lemma}

\begin{proof}
Let $K$ be generated by an element $a$. Then by assumption that $K$ is distorted, we have that for any $d$ there exists a $r \geq 0$ so that $\Delta_K^G(r)>dr$. Letting $\Delta_K^G(r)=\max\{ |m|: |a^m|_G \leq r\} = |m_0|$ for some integer $m_0 =m_0(r)$ we have that $|a^{m_0}|_G \leq r \leq \frac{|m_0|}{d}$. Rephrasing, we may say that for every $\varepsilon >0$, we can find $a^{m}$ so that $|a^m|_G \leq \varepsilon m$. Let us fix $m=m(\varepsilon)$. Consider any $a^r$. Write $r=km+l$ where $0 \leq l < m$. Let $c=c(m)=\max\{|a^l|_G: 0 \leq l < m\}$. Then $a^r=(a^{m})^ka^l$ and so $|a^r|_G \leq k|a^m|_G+c \leq r \varepsilon +c$. Because $r$ was arbitrary, it follows that the relative growth function of $K$ is at least $\varepsilon^{-1}r+C$ for $C=C(\varepsilon)$. Because $\varepsilon$ was arbitrarily small, the relative growth function is not bounded from above by any linear function.
\end{proof}

\begin{lemma}\label{cm3}
If $H$ is a finitely generated subgroup of a finitely generated group $G$, then $g_H^G(r) \approx r$ implies that $\Delta_H^G(r)$ must also be equivalent to a linear function. That is to say, if the embedding is distorted, then the relative growth is non-equivalent to any linear function. Moreover, if $g_H^G(r)$ is $o(r^2)$, then $H$ is virtually infinite cyclic.
\end{lemma}

\begin{proof}
By assumption, and by Equation \eqref{ng}, $g_H^G(r) \succeq g_H(r)$, we have that $g_H(r)$ is equivalent to a linear function (respectively, bounded above by $r^2$). Therefore, by Gromov's Theorem, we have that $H$ is virtually (finitely generated) nilpotent. So $H$ has a finite index (finitely generated) nilpotent subgroup $M$. $M$ is embedded in a direct product of a finite and a (finitely generated) torsion-free nilpotent group $T$, by \cite{baumslag}. The growth functions of $H$, $M$ and $T$ are equivalent, by Lemma \ref{findex}. Thus, by Bass's formula, in the case where the relative growth is linear, we have that $1=\textrm{rk}(H/H')$, which implies that $H'$ is finite and $H$ is virtually cyclic: there exists an infinite cyclic subgroup $K$ of $H$ with finite index. Suppose by way of contradiction that the embedding of $H$ to $G$ is distorted. Then because $[H:K]<\infty$, the embedding of $K$ to $G$ is also distorted by Lemma \ref{findex}. This implies that the function $\frac{\Delta_K^G(r)}{r}$ is unbounded. Therefore, because $K$ is cyclic, it follows from Lemma \ref{erlr} the relative growth of $K$ in $G$ also has $\frac{g_K^G(r)}{r}$ unbounded. This is a contradiction to the hypothesis that the relative growth of $H$ in $G$ is equivalent to a linear function. Now, in the case where the relative growth is $o(r^2)$, we have by Bass's formula, Equation \eqref{bass}, and the fact that the growth rate is 2, either $T$ is free abelian of rank 1 or $T$ is free abelian of rank 2. If $T$ is free abelian of rank 2, then the growth and relative growth of $T$ and therefore of $H$ are at least quadratic, by equations \eqref{ng} and \eqref{bass}, which is a contradiction to the hypothesis. Therefore $T$ is infinite cyclic. In this case, $M$ is also virtually $\z$ and so is $H$.
\end{proof}

\begin{lemma}\label{ygs}
Let $G$ be a finitely generated group, and let $H$ be an infinite cyclic subgroup generated by an element $a \in G$. Then $$\Delta_H^G(r) \succeq g_H^G(r).$$
\end{lemma}

\begin{proof}
We have that $$\Delta_H^G(r)=\max\{|a^k|_H : |a^k|_G \leq r\}=\max\{|k|:|a^k|_G \leq r \}=k_0$$ for some integer $k_0=k_0(r)$. Then if $|a^m|_G \leq r$ we have that $|m| \leq k_0,$ by the definition of $k_0$. That is, if we consider the set $$S=S(r)=\{1, a, a^{-1}, a^2, a^{-2}, \dots, a^{k_0}, a^{-k_0}\}$$ we have that $\{a^m \in H : |a^m|_G \leq r\} \subseteq S$. Therefore, $$g_H^G(r) \leq \#S = 2k_0+1 \approx \Delta_H^G(r)$$ as required.
\end{proof}

Note that for any infinite subgroup $H$ of a finitely generated group $G$, where $H$ is not locally finite, the relative growth function $g_H^G$ of the embedding is at least linear up to equivalence, as we will prove in Theorem \ref{2.2} Part $(1)$.

Combining Lemmas \ref{erlr} and \ref{ygs}, we have proved the following.

\begin{prop}\label{ldr}
An infinite cyclic subgroup of a finitely generated group is undistorted if and only if it has relative growth function equivalent to a linear function.
\end{prop}

\begin{rem}
Observe that Proposition \ref{ldr} does not hold in a larger class of groups than virtually infinite cyclic. Indeed, it suffices to let $H$ be finitely generated with growth function greater than linear, up to equivalence and consider the identity embedding of $H$ to itself. Moreover, the same example shows that Lemma \ref{ygs} cannot hold in a larger class of subgroups than virtually infinite cyclic. 

\end{rem}

\section{Relative Growth Functions for Subgroups}\label{ekh}

In this section we will may make use of Theorems \ref{olshd}, and \ref{olshc} to prove Theorem \ref{2.2}, a result for classifying which functions are equivalent to relative growth functions of cyclic subgroups in finitely generated groups. We begin with proving the necessary conditions.

\begin{proof} of Theorem \ref{2.2} Part $(1)$:
%Item $(a)$ is clear because $\#(B_G(r)) \preceq 2^r$ for any group $G$.

Let $G$ be generated by a finite set $X$ and suppose that $H$ has an infinite finitely generated subgroup $K$. Then there exists $c>0$ and a sequence in $K$ given by ${\bf h'} = \{h_i'\}$ such that $|h_{i+1}'|_X-|h_i'|_X \leq c$ for all $i$ and $|h_i'|_X \rightarrow \infty$.

Every function $f: \mathbb{N}  \to \mathbb{N} $ has a superadditive  closure $\bar f$ given by the following formula. For $r \in \mathbb{N} $, consider a partition $P$ of $r=n_1+ \cdots + n_s$ where $n_1, \dots, n_s$ are positive integers. One defines $$\bar f(r) = \max_P (f(n_1)+...+f(n_s)),$$ where the maximum is taken over all such partitions $P$ of $r$.

It suffices to prove that the relative growth function $g=g_H^G$ of $H$ in $G$ is  $\sim$-equivalent to its superadditive  closure $\bar g.$ The inequality $g(r)\le \bar g(r)$ is obvious.

To obtain an upper bound for $\bar g(r)$ we will estimate $g(n_1)+...+g(n_s)$ for an arbitrary partition of $r=n_1+\dots+n_s.$ Choose a finite subsequence $\{h_1=1, h_2,\dots, h_s\}$ of ${\bf h'}$ with $$n_1+2\sum_{k=2}^{i-1} n_k+n_i+(i-2)c<|h_i|_X\le n_1+2\sum_{k=2}^{i-1} n_k+n_i +(i-1)c \;\; for \;\; i\ge 2.$$

Let $i \in \{1, \dots, s\}$. Let $B_i$ be the ball in $G$ of center $h_i$ and radius $n_i$. Note that the distance in $G$ between $h_{i}$ and $h_{j}$ is at least $|h_j|_X-|h_i|_X>n_i+n_{j}$ if $1\le i<j\le s,$ and so the balls $B_1,\dots, B_s$ are pairwise disjoint. Every $B_i$ has exactly $f(n_i)$ elements from $H$ since each of them is centered at an element of $H$, and by definition of $g$ as the relative growth function. Now let $x \in B_i$. Then $|x|_X \leq |h_{i}^{-1}x|_X+|h_i|_X \leq n_i +n_1+2\sum_{k=2}^{i-1}n_k+n_i+(i-1)c$. Therefore all the balls  $B_1,\dots, B_s$ are contained in the ball $B$ of radius $R=n_1+2\sum_{k=2}^s n_i  +(s-1)c < 2r +sc\le (c+2)r$ centered at $1.$ Hence $g(n_1)+\cdots+g(n_s)\le g((c+2)r).$ Since this is true for arbitrary partition $r=n_1+\dots+n_s,$ we have $\bar g(r)\le g(Cr),$ where $C=c+2$ does not depend on $r,$  and so $\bar g(r) \sim g(r).$ 

\end{proof}

Now we proceed with the proof of Theorem \ref{2.2} Part $(2)$, which contrasts with Part $(1)$ of the theorem in case the subgroup is locally finite.

\begin{proof} of Theorem \ref{2.2} Part $(2)$:

Let $F: \mathbb{N} \rightarrow \mathbb{N}$ be any nondecreasing function with $\displaystyle\lim_{r \to \infty} F(r) = \infty$. We will show that there exists an embedding of $H$ into a finitely generated group $G$ so that the relative growth function satisfies $g_H^G \prec F$. In particular, this implies that $g_H^G$ is not equivalent to any superadditive  function, else we would have a contradiction by choosing $F(r) = r$, as any non-zero superadditive  function dominates a linear one.

Because $H$ is locally finite and countably infinite, we may write $H=\displaystyle\bigcup_{i=1}^{\infty} H_i$, where for each $i  \geq 1$, we have $H_i=\langle h_1,\dots, h_i \rangle$ for elements $h_1, \dots, h_i \in H$ with $\# H_i = n_i<\infty$ and $H_{i+1}>H_{i}$.

We will define a function $l:H \rightarrow \mathbb{N} $ so that $l$ satisfies the $(D)$ condition. First, we assign a length $\tilde{l}(h_i^{\pm 1})=l_i$ to every $h_i^{\pm 1}$ so that 

$\min(l_i, F(\sqrt{l_i}))>n_i$ and $l_{i+1} > l_i$.
The set $S=\{h_1^{\pm 1},h_2^{\pm 1},\dots\}$ generates $H$. Therefore, for $h \in H$ we have several possible expressions of the form 
\begin{equation}\label{hpsm}
h=h_{j_1} \cdots h_{j_s} \textrm{ where } h_{j_1},\dots,h_{j_s} \in S
\end{equation}
For any  expression $P$ of this form, we define $$l(h,P) =l_{j_1}+\cdots+l_{j_s}.$$ Finally, for $h \in H$, we let $$l(h) = \min_P l(h,P)$$ where $P$ ranges over all possible expressions \eqref{hpsm}. It is clear that $l(h_i) \leq l_i$. Because $h_i\notin H_{i-1}$, arbitrary factorization $P$ of $h_i$ must have a factor $h_j^{\pm 1}$ with $j \geq i$ whence $l(h_i,P) \geq l_j \geq l_i$, so $l(h_i) \geq l_i$ as well. Hence $l(h_i)=l_i$.

The conditions $(D1)$ and $(D2)$ follow from the definiton of $l$. We will show that $(D3)$ holds. We must estimate $\#\{h \in H: l(h)\le r\}$, for $r>0$. We may assume that $l_{i}\le r< l_{i+1}$ for some $i$. Then if $h \in H$ has $l(h) \leq r$, we have that $h$ belongs to $H_i.$ Hence $\#\{h: l(h)\le r\}\le n_i\le l_i\leq r ,$ and $(D3)$ follows.

Now using Theorem \ref{olshd} we may embed $H$ in a finitely generated group $G$ with $l(h)\le C|h|_G $ for some $C\ge 1$ and all $h\in H$. Choosing $Cr\in [l_i,l_{i+1})$ for some $i$, we have by use of the fact that $F$ is nondecreasing that $$g_H^G(r) = \#\{h \in H : |h|_G \leq r\} \le \#\{h \in H: l(h)\le Cr\}$$ $$\le n_{i} < F(\sqrt{l_i})\le F(\sqrt{Cr})\le F(r) $$ if $i$ is large enough, i.e. $g_H^G \prec F(r)$, as required.
\end{proof}

Finally, we can add that the necessary conditions are also sufficient, but only for infinite cyclic groups.

\begin{proof} of Theorem \ref{2.2} Part $(3)$:
Define a function $l: \z \rightarrow \{0,1,2,\dots,\} $ as follows. If $r>0$ then $l(r)=m$ where $m$ is the minimal number so that $f(m) \geq r.$ If $r<0$ let $l(r)=l(-r)$. Finally, let $l(0)=0$. We will verify that $l$ satisfies the $(C)$ condition of Theorem \ref{olshc}. The symmetry condition $(C1)$ is satisfied by definition of $l$. Next we show that $l$ satisfies $(C2)$; i.e. that it is subadditive. Let $r,s \in \z$ and $l(r+s)=k$. Then $k$ is mimimal such that $f(k) \geq |r+s|.$ Let $l(r)=R$ and $l(s)=S$, so $f(R) \geq |r|$ and $f(S) \geq |s|$. Suppose by way of contradiction that $k>R+S$. Then $f(R+S)<r+s$. But $f$ is superadditive, so $f(R+S) \geq f(R) + f(S) \geq |r|+|s| \geq |r+s|$, a contradiction. The condition $(C3)$ is satisfied because for $n \in \mathbb{N}$ we have that $\#\{r \in \mathbb{N} : l(r) \leq n\} \approx f(n)$. This follows because for $r>0$, we have by the definition of $l$ that $f(n) < r$ if and only if $l(r) >n$. Therefore, $\#\{r : r>0, l(r) \leq n\} = \#\{r:r>0,r \leq f(n) \}= f(n).$ Also, we have by hypotehsis that $f(n)$ is at most equivalent to an exponential function $2^n$. Therefore by Theorems \ref{olshd}, \ref{olshsolv} there exists a finitely generated (solvable) group $G$ generated by a set $T$ and an embedding of $H$ into $G$ so that the function $l$ is $\Theta$-equivalent to the length function on $G$, $g \mapsto |g|_{T}$, restricted to $H$. It remains to observe that $g_H^G(r) = \# \{h \in H : |h|_T \leq r\} \approx \# \{ h \in H : l(h) \leq r\} \approx f(r)$.
\end{proof}

It is worthwhile to note that Theorem \ref{2.2} Part $(3)$ is false even for non-virtually cyclic groups because we could not choose $f$ to be linear. Therefore, we ask the following question.

\begin{question}
Is it possile to find necessary and sufficient conditions for a function to be a relative growth function of an embedding of arbitrary finitely generated (not necesssarily cyclic) subgroup of a finitely generated group?
\end{question}

\section{Possible Distortion of Infinite Finitely Generated Groups}\label{pdicg}
In this section we will prove Theorems \ref{s12} and \ref{ellp}.

To setup the proof of Theorem \ref{s12}, part $(1)$, let $f$ be an increasing superadditive  function on $\mathbb{N}$ and $H$ an infinite group with finite generating set $Y$.
We will define an infinite sequence of elements $\{h_i\}$ in $H$ as well as sequences of natural numbers $\{l_i\}$ and $\{n_i=f(l_i)\}$. Let $l_1=2,$ so $n_1=f(l_1)$. Inductively, we let $l_{i+1}$ be the minimal number such that $f(l_{i+1})\ge 2f(l_i)$ and so
$n_{i+1}=f(l_{i+1})$ for $i\ge 1$. Finally, we choose $h_i \in H$ so that $|h_i|_H= n_i.$ The sequence is defined since $f$ is increasing. Since $f$ is superadditive , we have $f(l_i+l_i)\ge 2f(l_i),$ and so by minimality $l_i<l_{i+1}\le 2l_i$. We assign a weight $w(h_i^{\pm 1})=l_i.$

Every element $h \in H$ has many factorizations $P$ of the form  $h=x_1 \cdots x_k$, where $x_i \in Y^{\pm 1}$ or $x_i=h_j^{\pm 1}$ for some $j$. We define a function $l: H \rightarrow \mathbb{N}$ by setting, for any $h \in H$, \begin{equation}\label{ps} l(h) = \min_P \sum w(x_i) \end{equation} over all the factorizations of $h$ where $w(x_i)=1$ if $x_i \in Y^{\pm1}$. 

\begin{lemma}\label{cait}
The function $l$ defined in Equation \eqref{ps} satisfies the $(D)$ condition.
\end{lemma}

\begin{proof}
It is clear that $l$ satisfies the conditions $(D1)$ and $(D2)$. To show that $l$ also satisfies $(D3)$, we will count the number of elements  $h \in H$ which satisfy $l(h)= R\le r$. To estimate, we partition each $R \le r$ as $r_1+...+r_k$ with the summands either equal to $1$ or of the form $l_i$, where $i \ge 1$.  Assign to every summand $r_i$, a label from $Y^{\pm 1}$ (if $r_i=1)$ or  $\pm$ (if $r_i=l_j)$, the sign depending on whether the factor is $h_j$ or $h_j^{-1}$. 
The number of partitions of $R$ is bounded by $a^r$, where $a=1+\max \{j : l_j \leq r\}$. Thus the number of different labelings of a fixed partition by the finite alphabet is bounded by $(2|Y|+2)^r$.

Since a labeled partition determines the product $P$ (because all the numbers $l_1,l_2,\dots$ are different and greater than one), the number of elements  $h$ with $l(h) \leq r$ is bounded by $$\displaystyle\sum_{R=1}^r \#\{ h \in H : l(h) = R\} = ra^{r}(2|Y|+2)^r \leq (a(2|Y|+2))^{2r}.$$ This proves that (D3) holds.  
\end{proof}

By Theorems \ref{olshd} and \ref{olshsolv}, $H$ embeds in a finitely generated $G$ (solvable if $H$ is solvable) such that $l(h) \approx |h|_G$ for $h \in H$ and so that the distortion of $H$  in $G$ is equivalent to  $d(r)=\max\{|h|_H : l(h)\le r\}$. 

\begin{lemma}\label{wk}
We have that  $d(r) \approx f(r)$.
\end{lemma}

\begin{proof}
Assume that $r>1$. Then  $l_i \le r < l_{i+1}\le 2l_i$ for some $i$. Let $l(h)\le r$, where $h$ is chosen so that $d(r) = |h|_Y$. Suppose that the minimal factorization $P$ of  $h$ has  $k_j$ factors of length $l_j$ in $H$, for $j \in \mathbb{N}$. 
Then $r\ge l(h)=k_il_i+\cdots+k_1l_1 +k_0,$
and so $|h|_H \leq  k_i n_i+\cdots+k_1 n_1+k_0 = k_i f(l_i)+\cdots+k_1f(l_1)+k_0 \le f(k_il_i+\cdots+k_1l_1+k_0) \le f(r)$.  Therefore  $d(r)\le f(r).$

On the other hand, $d(l_i)\ge |h_i|_H = n_i = f(l_i)$ since $l(h_i) \le l_i$.  Therefore 
$d(r) \ge d(l_i) \ge f(l_i) \ge f([r/2+1]).$  Therefore  $f(r) \approx d(r)$, as desired.   
\end{proof}

Now Theorem \ref{s12} part $(1)$ follows from Lemmas \ref{cait} and \ref{wk}. Theorem \ref{s12} Part $(2)$ follows by Lemma \ref{hfd}.\\

To setup the proof of Theorem \ref{s12}, part (3), we will use one of Philip Hall's groups: in \cite{hall} it was proved that every countable abelian group $C$ is isomorphic to the center of a solvable (center-by-metabelian) group $G$ with two generators. We choose $C$ to be the countable direct product of groups of order two. That is, we let $H$ be a two-generated group with generating set $Y$ with infinitely generated central subgroup $C$ with generating set $\{c_i : i \in \z\}$ so that the order of each $c_i$ is $2$. Let $g: \mathbb{N} \to \mathbb{N}$ be an increasing function. 

By induction, we will define a sequence of central elements $\{h_i\}$ of $H$ depending on $g$, and an increasing sequence of positive integers $\{d_i\}$. Let $h_1=c_1$ and $d_1=1$. Then 
set $d_i=|h_1|_Y+\dots+|h_{i-1}|_Y$, and let $h_i$ be an element of
minimal $Y$-length in $C$ such that  $|h_i|_Y\ge  g(d_i)$; so in particular, $d_i\ge i$.

To define a length function $l$ on $H$, we prescribe a weight $d_i=w(h_i)$
to each $h_i^{\pm 1}$, and $w(x_1)=1$ if $x_i \in Y^{\pm1}$. Then for any $h \in H$, we have many factorizations $P$ of $h=x_1\cdots x_k$ where $x_i \in Y^{\pm 1}$ or $x_i=h_j^{\pm1}$ for some $j$.  We define a length function on $H$ as in \eqref{ps}: for any $h \in H$, 
\begin{equation}\label{pct}
l(h) = \min_P \displaystyle\sum_i w(x_i). 
\end{equation}

\begin{lemma}\label{da}
The function $l$ defined in Equation \eqref{pct} satisfies the $(D)$ condition.
\end{lemma}

\begin{proof}
Again the conditions (D1) and (D2) are clearly satisfied by $l$, and we want to check (D3). Assume that $l(h)\le r$. Since the elements $h_i$ are all central and have order $2$, a minimal factorization of $h$ has the form $z_1\dots z_s v$, where each $z_i$ is either $h_i$ or $1$, and
$v$ is a word of length at most $r$ in $Y^{\pm 1}$. Therefore the number of possible $v$'s
is bounded by $4^r$, and the number of products $z_1\dots z_s$ is at most
$2^s$. Here $s < r$  since otherwise $r\ge d_s\ge s>r$, a contradiction. Thus the number of elements $h\in H$ with $l(h)\le r$ is at most $8^r$, which is sufficient for (D3).
\end{proof}

By Theorems \ref{olshd} and \ref{olshsolv} we can embed the group $H$ into a finitely generated (solvable if $H$ is solvable) group $G$, such that the distortion is equivalent to the function $f(r) = \max \{|h|_Y :  h\in H, l(h)\le r\}$.

\begin{lemma}\label{vd}
If the function $g$ is bounded by a linear function on some infinite sequence, but grows faster than any recurisve function, then the distortion function of $H$ in $G$ is not bounded by a recursive function, yet it is bounded by a linear function on an infinite subset of $\mathbb{N}$.
\end{lemma}

\begin{proof}
The embedding is distorted since $f(d_i)\ge g(d_i)$ for the increasing sequence $d_1,d_2,\dots$ because $l(h_i)\le d_i$ while $|h_i|_Y\ge g(d_i)$.  The distortion is not recursive if $g$ grows faster than arbitrary recursive function.

It suffices to show that the function  $f$ is bounded by a linear function on an infinite subset of $\mathbb{N}$, since $f \approx \Delta_H^G$. For this goal we consider an element $h\in H$ with $l(h) \le d_i-1$ for a fixed $i$ and its minimal factorization $h=z_1\dots z_s v$ as above.  It follows from the definition of weights, that none of $z_1,\dots, z_s$  belongs to $\{h_i, h_{i+1},\dots\}$, and so $|z_1\dots z_s|_Y\le |h_1|_Y+\dots+ |h_{i-1}|_Y\le d_i$. Also we  have $|v|_Y\le d_i-1$ since $l(h)\ge |v|_Y$ by the definition of $l$ and $v$. Hence $|h|_Y\le 2d_i-1,$
and so $f(r)\le 2r+1$ on the increasing infinite sequence  of the numbers $r=d_i-1$, as desired.
\end{proof}

\begin{rem}\label{bjo}
If a function $F$ on $\mathbb{N}$ is bounded by a linear function on an infinite subset of $\mathbb{N}$ and is equivalent to a superadditive function, then it is linear up to equivalence.
\end{rem}

\begin{cor}\label{spra}
The distortion function $\Delta_H^G$ is not equivalent to a superadditive function.
\end{cor}

\begin{lemma}\label{emc}
There exists a finitely presented group $K$ containing the group $H$ as a distorted subgroup such that $\Delta_H^K$ is bounded by a linear function on an infinite subset of $\mathbb{N}$, but is not bounded by a linear function on the entire set $\mathbb{N}$.
%hence by Remark \ref{bjo} is not equivalent to a superadditive function.
\end{lemma}

\begin{proof}
It is easy to see that the function $l$ of Equation \eqref{pct} is computable on $H$ if $g$ is recursive when $H$ has algorithmically decidable word problem. To see that $H$ can be constructed to have solvable word problem, we will briefly give the explicit construction of the group $H$. First we define a group $B$ by generators and relations :
$B=\langle b_i, i\in \z \mid b_i^2=1, [[b_i,b_j],b_k]=1, i,j,k \in  \z \rangle$.
The commutators  $c_{ij}=[b_i,b_j]$ ($i<j$) generate a central subgroup
in $B$ and they are linearly independent over $\z/2\z$.
Therefore adding the relation $c_{ij}=c_{kl}$ for
all $i,j,k,l$ with $j-i=l-k$, we obtain  a group (we keep the same notation
$B$) whose central subgroup $C$ is independently generated by the commutators
$c_m$, $m=1,2,\dots$, where $c_m= c_{i,i+m}=[b_i,b_{i+m}]$ for every $i\in \z$.
By Dyck's theorem, the mapping $b_i\mapsto b_{i+1}$ ($i\in \z)$ extends to an
automorphism $\alpha$ of $B$ leaving fixed all the elements of $C$. Hence the extension
$H$ of $B$ by this automorphism is generated by two elements $b_0$ and $\alpha$
and has the infinite central subgroup $C$ generated by all $c_m$-s, the elements of order $2$. Assume now that $g(x)=x^2$, and so $l$ is computable on $H$. By \cite{olshr}, one can choose a finitely presented group $K$ containing $H$ such that $l(h) = \Theta(|h|_X)$, where $X$ is a finite set of generators for $K$. Then as in Lemma \ref{vd}, we have that the distortion function of $H$ in $K$ is not bounded by a linear function but bounded by a linear funtion on an infinite subset of $\mathbb{N}$.
%Therefore by \cite{olshr} one can choose the group $K=G$ to be finitely presented.

\end{proof}

Theorem \ref{s12}, parts $(3)$ and $(4)$ follows from Lemmas \ref{vd}, \ref{emc} and Corollary \ref{spra}.\\

%Now we move on to prove Theorem \ref{ellp}

To prove Part $(2)$ of Theorem \ref{ellp}, let $H=\z$, and consider the function $l$ defined in \eqref{ps}. We have that the relative growth function of the induced embedding is by definition $\#\{i \in \z: |i|_G \leq r\}$ which is equivalent to $\# \{i \in \z : l(i) \leq r\}$. 

\begin{lemma}\label{dd}
For any $r \geq 0$, the number of positive integers $i$ with $l(i) \leq r$ is bounded from above by a function equivalent to $2^{\sqrt{r}}.$
\end{lemma}

\begin{proof}

Let $l(i)\leq r$ where the value of $l(i)=l_{i_1}+\cdots+l_{i_k} + n$ arises from the partition $P$ of $i=\pm n_{i_1} \pm \cdots \pm n_{i_k} \pm n$. Consider the number of summands in the expression $l(i)=l_{i_1}+\cdots+l_{i_k} + n$. 

It does not exceed the number $r p(r)$, where $p(r)$ is the number of partitions of $r$ in positive summands. By \cite{partition}, $p(r) \preceq 2^{\sqrt{r}}$, and so the number of expressions $i=\pm n_{i_1} \pm \cdots \pm n_{i_k} \pm n$ with $l(i)\le r$ is at most $ r 2^{\sqrt{r}}$, up to equivalence. Finally, $ r 2^{\sqrt{r}}
\approx 2^{\sqrt{r}}.$

  Now we estimate the number $m$  of pairwise distinct summands  in the
  expression  $i=\pm n_{i_1} \pm \cdots \pm n_{i_k} \pm n$.  Since one has $1+2+...+m > r$ for $m \ge \sqrt{2r}$,
  we must have $m<\sqrt{2r}$. Hence the number of possible sign arrangements
  in $\pm n_{i_1} \pm \cdots \pm n_{i_k} \pm n$ is less than  $2^{\sqrt{2r}}.$
  This implies that the number of positive integers  $i$ with  $l(i)\le r$
  is bounded (up to  equivalence) by $ 2^{\sqrt{r}}  2^{\sqrt{2r}}
 \approx  2^{\sqrt{r}}.$

\end{proof}

\begin{proof} of Theorem  \ref{ellp} Part $(1)$:\\

Assume that $H$ is an infinite, finitely generated, exponentially distorted subgroup of a finitely generated group $G$. Then for every $k\geq 1$ we can find, using the exponential distortion, elements $g_k \in H$, such that $3 |g_{k-1}|_H < |g_k|_H$ but $|g_k|_G \leq Dk$ for some $D > 0$. Then the ball of radius $Dk^2$ in $G$ contains all the products $g_k^{\nu_k} \cdots g_1^{\nu_1}$ where each $\nu_j \in \{0,1\}$ because any such product has length at most $|g_k^{\nu_k} \cdots g_1^{\nu_1}|_G \leq |g_k|_G+\dots+|g_1|_G \leq D(\frac{k(k+1)}{2})$. Moreover, all these products are pairwise distinct. 

To see this, suppose that we did have two equal expressions $$g_k^{\nu_k} \cdots g_1^{\nu_1} = g_k^{\nu'_k} \cdots g_1^{\nu'_1}.$$ Without loss of generality, $\nu_{k}, \nu'_{k} \ne 0$ and $\nu_k \ne \nu'_k$. 

Then $g_k$ is expressible as a product of $g^{\pm 1}_1,..., g_{k-1}^{\pm 1 }$,
  where each of $g_j$ occurs at most twice. Therefore
  $|g_k|_H\le 2|g_{k-1}|_H+\cdots+2|g_1|_H< |g_k|_H(\frac{2}{3}+\frac{2}{9}+...+\frac{2}{3^{k-1}})<|g_k|_H$,
  a contradiction.

Therefore, the products $g_k^{\nu_k} \cdots g_1^{\nu_1}$ represent $2^k$ distinct elements. Therefore for the relative growth function of $H$ in $G$, $g_H^G$, we have inequalities $g_H^G(Dk^2) \geq 2^k$, whence $g_H^G(r) \succeq 2^{\sqrt{r}}$.

\end{proof}

\section{Connections Between Relative Growth and Distortion}\label{swz}

\subsection{Constructing Length Functions with Prescribed Data}
We begin by introducing some lemmas that will be used in the proof of Theorem \ref{cm}.

\begin{lemma}\label{seq}
There exist increasing sequences $\{a_i\}_{i \in \mathbb{N}}, \{n_i\}_{i \in \mathbb{N}}$ of natural numbers satisfying the following properties for all $i \geq 2$.
\begin{itemize}
\item $a_1=n_1=1$
\item $a_i \ge  2^{i+2}n_{i-1}$ 
\item $ n_i > n_{i-1}a_i/a_{i-1}$
\item $n_{i-1} \mid n_i$
\item There does not exist a recursive function $f$ such that $n_i \leq f(a_i)$ for every $i$.
\item $a_i > a_{i-1}$
\end{itemize}
\end{lemma}

\begin{proof}
We use that the set of recursive functions is countable. Denote it by $\{f_i\}_{i \in \mathbb{N}}$. Suppose that $a_{i-1}$ and $n_{i-1}$ have been defined. Let $a_i=2^{i+3}n_{i-1}+1$. Let $$n_i=n_{i-1} \max \{\lceil\frac{n_{i-1}a_i}{a_{i-1}}\rceil +1, \max_{j\leq i}(f_j(a_i))+1\}.$$
\end{proof}

Define a function $l: \mathbb{Z} \rightarrow \{0,1,2,\dots,\} $ by the formula $l(0)=0$ and for nonzero $n \in \z$
\begin{equation}\label{ellfn}
l(n) = \min\{a_{i_1}+\cdots+a_{i_s} | n=\pm n_{i_1} \pm \cdots \pm n_{i_s} \textrm{ for some } i_1,\dots,i_s \in \mathbb{N} \}
\end{equation}

We will show that the function $l$ of equation \eqref{ellfn} satisfies the $(C)$ condition. The following auxilliary Lemmas will help us to prove that condition $(C3)$ is satisfied. Let $r$ be a natural number. We want to be able to compute $\#\{n \in \mathbb{Z} :l(n) \leq r\}.$ Suppose that $n$ is in this set, let $$l(n) = a_{i_1}+\cdots+a_{i_s}$$ and consider the corresponding minimal presentation given by $$n=\pm n_{i_1} \pm \cdots \pm n_{i_s}.$$ 

\begin{lemma}\label{rrr}
This expression has no summands with subscript greater than or equal to $j$, where $j$ is defined by the property: $a_{j-1}\le r <  a_j$. 
\end{lemma}

\begin{proof}
This is true, since otherwise $l(n) \ge a_j > r,$ a contradiction.
\end{proof} 

Therefore, we may rewrite the expression as $$n=k_1n_{1} + \cdots + k_{j-1}n_{j-1},$$ for some integers $k_1, \dots, k_{j-1}$. We may assume that $j \geq 3$ in the following, since eventually we will let $r$ become very large, and with it, so will $j$.

\begin{rem}\label{lsl}
Observe that $|k_{i}| \leq \frac{r}{a_{i}}$, for $i \leq j-1$ for otherwise, $l(n) \geq |k_{i}|a_{i} >r$.

\end{rem}

\begin{lemma}\label{atpy}
For any $2 \leq i < j-1$ we have that $|k_i| < \frac{n_{i+1}}{n_i}$. 
\end{lemma}

\begin{proof}
Suppose by way of contradiction that there is $s \geq 0$ such that $|k_{i}| = \frac{n_{i+1}}{n_{i}}+s$. Then $$n=k_1n_1+\cdots \pm |k_{i}|n_{i}+\cdots+k_{j-1}n_{j-1}=$$ $$k_1n_1+\cdots \pm \bigg(\frac{n_{i+1}}{n_{i}}\bigg)n_{i} \pm sn_{i}+\cdots+k_{j-1}n_{j-1}=$$ $$k_1n_1+\cdots \pm sn_{i} + (k_{i+1} \pm 1)n_{i+1}\cdots + k_{j-1}n_{j-1}.$$ Thus we have two expressions for $n$ as in Equation \eqref{ellfn}. By hypothesis, $l(n) = |k_1|a_1+\cdots+|k_{j-1}|a_{j-1}$ and so since this is the minimum over all such expressions for $n$, we have that $|k_i|a_i+|k_{i+1}|a_{i+1} \leq sa_i+|k_{i+1} \pm 1| a_{i+1}$. By the triangle inequality, $sa_{i}+|k_{i+1} \pm 1|a_{i+1} \leq sa_{i}+ |k_{i+1}|a_{i+1}+a_{i+1}$. Thus we have that $|k_i|a_i \leq sa_i+a_{i+1}$. Using again that $|k_i|=\frac{n_{i+1}}{n_i}+s$ we have that $\frac{n_{i+1}a_i}{n_i} \leq a_{i+1}$, contrary to the third bullet point of Lemma \ref{seq}.

 \end{proof}

\begin{lemma}\label{fpt}
The function $l$ defined in Equation \eqref{ellfn} satisfies the $(C)$ condition of Theorem \ref{olshc}. Moreover, $\#\{n : l(n) \leq r\}$ is $o(r^2)$.
\end{lemma}

\begin{proof}
Observe that for each $n \in \z, n=n \cdot 1 = n \cdot n_1$, so the function is defined. To see that the condition $(C1)$ is satisfied, select $n \in \z$. Without loss of generality, $n \ne 0$. Let $l(n) = a_{i_1}+\cdots+a_{i_s}$, so that there is an expression $n=\pm n_{i_1} \pm \cdots \pm n_{i_s}.$ This implies that $-n=-(\pm n_{i_1} \pm \cdots \pm n_{i_s})$ and so by definition of $l$, we have that $l(-n) \le l(n)$. Equality holds by symmetry. The $(C2)$ condition is similarly easy: let $l(n)=a_{i_1}+\cdots+a_{i_s}$ and $l(m)=a_{j_1}+\cdots+a_{j_t}.$ Then one expression representing $n+m$ is $\pm n_{i_1} \pm \cdots \pm n_{i_s} \pm n_{j_1} \pm \cdots \pm n_{j_t}$ which implies that $l(n+m) \leq l(n)+l(m)$. Finally, we will show that the condition $(C3)$ is satisfied by showing that $\#\{n \in \z: l(n) \leq r\}$ is $o(r^2)$ (and using the fact that $r^2 \preceq 2^r$). Taking into account the signs, we have by Remark \ref{lsl} and Lemma \ref{atpy} that the number of values of $n$ with $l(n) \le r$ is at most the product over the number of values of $k_j$, namely: $$\frac{3r}{a_{j-1}}\frac{3r}{a_{j-2}}\frac{2n_{j-2}}{n_{j-3}}\frac{2n_{j-3}}{n_{j-4}}\cdots\frac{2n_{2}}{n_1}<\frac{9r^22^{j-3}n_{j-2}}{a_{j-1}a_{j-2}}<\frac{r^22^{j+1}n_{j-2}}{a_{j-1}a_{j-2}}<\frac{r^2}{a_{j-2}},$$ by the choice of $a_i$ in Lemma \ref{seq}. Now we have that $\displaystyle\lim_{r \to \infty} a_{j-2} = \infty$ by the choice of $j=j(r)$ as in Lemma \ref{rrr}. Therefore, $\#\{n : l(n) \leq r\}$ is $o(r^2)$.

\end{proof}

 \subsection{Producing Bounds on Relative Growth}

We now introduce some notation and lemmas towards proving Theorem \ref{cm} Part $(2)$ in the case where the finitely generated subgroup $H$ is cyclic.

That is, we aim to prove the following Lemma.

\begin{lemma}\label{cm2}
For any virtually infinite cyclic subgroup $H$ of a finitely generated group $G$ such that $g_H^G(r)$ is $o^e(r^2)$, we have that $\Delta_H^G(r)$ is bounded above by a recursive function.
\end{lemma}

Let $H = \langle a \rangle \leq G$ where $G$ is finitely generated. Without loss of generality, we may include $a$ in the finite set of generators of $G$. Consider the length function corresponding to the embedding given by $l : \mathbb{Z} \rightarrow \mathbb{N} : l(r) = |a^r|_G.$

\begin{lemma}\label{waffl}
Suppose that $\phi(r)$ is a function with effective limit infinity, and that the distortion function $\Delta_H^G(r)$ is not bounded from above by any recursive function. Then we cannot have $20 l(n) \ge \phi(n)$ for almost all $n$.
\end{lemma}

\begin{proof}
Suppose by way of contradiction that $20 l(n) \geq \phi(n)$ for all $n>N$ and some $N$. Then the effective limit of $l(n)$ is also infinity, and so given any $C$, one can effectively compute $N(C)$, such that $l(n) >C$ for any $n\ge N(C)$. But this means that the distortion function $\Delta_H^G(r)$ is bounded from above by the recursive function $N(r)$ of $r$, a contradiction.
\end{proof}

\begin{rem}\label{na}
By Lemma \ref{waffl}, there exists an infinite sequence $n_1=1< n_2< n_3<\dots$  such that $20 l(n_i)< \phi(n_i).$
In addition we may assume by choosing a subsequence that for all $i>1$ we have 
\begin{itemize}
\item $n_i>i^2(n_1+\cdots+n_{i-1})$. 
\item $n_{i+1} > 6n_i^2$
\item $a_{i+1} \geq a_i$ where $a_j=l(n_j)$ for $j=1,2,...$
\end{itemize}
\end{rem}
 
Fix $i \in \mathbb{N}$. Let $r=n_i$. Our current goal is to obtain a lower bound for the number of integers $n$ with $l(n) \leq r$. In order to compute this, we will consider for the moment only numbers $n$ of the form $n=k_1n_1+...+k_in_i$
where 
\begin{equation}\label{coeffs}
0\le k_j < \frac{n_{j+1}((j+1)^2-1)}{n_j(j+1)^2}, \textrm{ for } j \in \{1,\dots,i\}.
\end{equation}

\begin{lemma}\label{pbj}
Different coefficients satisfying the condition of Equation \eqref{coeffs} define diferent sums $k_1n_1+\cdots + k_in_i$.
\end{lemma}

\begin{proof}
This is true because otherwise for some $j \leq i$ and $m>0$, we will have $m n_j = m_{j-1}n_{j-1}+\cdots+m_1n_1$ where for each $s \leq j$, $|m_s| <(1-\frac{1}{(s+1)^2})\frac{n_{s+1}}{n_s}$ by the choice of coefficients in Equation \eqref{coeffs}. Then we have that $$m n_j < n_j\bigg(1-\frac{1}{j^2}\bigg)+ n_{j-1}\bigg(1-\frac{1}{(j-1)^2}\bigg) +\cdots + n_{2}\bigg(1-\frac{1}{(2)^2}\bigg) < $$ $$n_j\bigg(1-\frac{1}{j^2}\bigg)+\frac{n_j}{j^2}=n_j$$ by the choice of $n_j$, a contradiction.
\end{proof}

\begin{lemma}\label{wrl}
If we assume in addition to \eqref{coeffs} that, for $i>1$, $$k_{i-1}<\frac{n_i}{3n_{i-1}}, k_i\le \frac{r}{3a_i} \textrm{ and } r = n_i,$$ then we have that $l(n) \leq r$.
\end{lemma}

\begin{proof}
Observe that the restrictions on $k_{i-1}$ and $k_i$ in the statement of the Lemma are stronger than the initial assumptions in \eqref{coeffs}. 
By the properties of $l$, the definition of $a_i$, by inequiality $\eqref{coeffs}$ and together with the additional assumptions stated in Lemma \ref{wrl}, we have that $$l(k_1n_1+\cdots+k_in_i)\le k_1a_1+\cdots+k_ia_i \leq \frac{a_1 n_2}{n_1}+\cdots+\frac{a_{i-2}n_{i-1}}{n_{i-2}}+\frac{a_{i-1} n_i}{3 n_{i-1}}+\frac{a_i r}{3a_i}.$$ We clearly have that $\frac{a_i r}{3a_i} = \frac{r}{3}$.

Also, because $a$ is a generator of $H$, we have that $a_{j}=l(n_{j}) = |a^{n_{j}}|_G \le n_{j},$ for any $j \in \{1, \dots, i\}$. Finally, observe that $$\frac{a_{j-1} n_j}{n_{j-1}}\leq \frac{a_jn_{j+1}}{6n_j}$$ for each $j \geq 2$ by Remark \ref{na} as well as the fact that $a_{j-1} \leq a_j$ for all $j$. Therefore, $$\frac{a_1 n_2}{n_1}+\cdots+\frac{a_{i-2}n_{i-1}}{n_{i-2}}+\frac{a_{i-1} n_{i}}{3 n_{i-1}} \leq (\frac{1}{2^{i-2}}+\cdots+\frac{1}{2}+1)\frac{a_{i-1}n_i}{3n_{i-1}} < 2\cdot \frac{n_i}{3} = \frac{2r}{3}.$$ Therefore, $l(n) \leq \frac{2r}{3}+\frac{r}{3}=r$.
\end{proof}

We now proceed with the proof of Lemma \ref{cm2}.

\begin{proof}
Let $r=n_i$ be fixed. Lemmas \ref{pbj} and \ref{wrl} together with the choice of $k_1, \dots, k_i$ imply that the number of $n$'s with $l(n) \leq r$ is at least $$\left(1-\frac{1}{4}\right)\frac{n_2}{n_1}\left(1-\frac{1}{9}\right)\frac{n_3}{n_2}\cdots \left(1-\frac{1}{{(i-1)}^2}\right)\frac{n_{i-1}}{n_{i-2}}\left(\frac{1}{3}\frac{n_i}{n_{i-1}}\right)\frac{n_i}{3a_i} >$$ $$\prod_{j=2}^{i-1}\left(1-\frac{1}{j^2}\right)\frac{n_i}{3a_i} \frac{n_i}{3} =\prod_{j=2}^{i-1}\left(1-\frac{1}{j^2}\right)\frac{r^2}{9a_i} >\frac{r^2}{20a_i}.$$ This follows because $r=n_i$ and the product $\displaystyle\prod_{j=2}^ {\infty}(1-\frac{1}{j^2})$ converges to $\frac{1}{2}$. Hence the value of the corresponding relative growth function of $\z$ at $r = n_i$ is at least $$\frac{r^2}{20 a_i} = \frac{r^2}{20 l(n_i)} > \frac{r^2}{\phi(n_i)}$$ by the choice of $n_i$. Thus this function $g_{\z}^G(r)$ is not bounded from above by any $\frac{r^2}{\phi(r)}$ where $\phi$ is a function with effective limit infinity, because the $\phi$ we started with was arbitrary, a contradiction.
\end{proof}

Theorem \ref{cm} Part $(1)$ follows by Equation \eqref{ng} together with Lemma \ref{cm3}. Part (2) follows by Lemma \ref{cm2} in conjunction with Lemma \ref{cm3}.

%\begin{proof}
%Because $H$ is finitely generated, we have by Equation \eqref{ng} that $g_H^G(r) \succeq g_H(r)$. Therefore, $g_H(r)$ is also $o(r^2)$, which implies that $H$ has polynomial growth. By Gromov's theorem, $H$ is virtually nilpotent. By formula \eqref{bass}, $H$ must be virtually cyclic.
%\end{proof}

\begin{proof} of Theorem \ref{cm} Part $(3)$:

 We will show that there exists a cyclic subgroup $H$ of a two generated group $G$ such that $\Delta_H^G(r)$ is not bounded above by any recursive function, yet $g_H^G(r)$ is $o(r^2)$. 

Consider the function $l$ of Equation \eqref{ellfn}. Lemma \ref{fpt} and Theorem \ref{olshc} impliy that there is an embedding of a cyclic subgroup $\langle a \rangle$ into a finitely-generated (solvable) group  $G$ with $|a^n|_G=\Theta(l(a^n))$.

 By the choice of the sequences $\{a_i\}, \{n_i\}$ the embedding has distortion function which is not bounded by a recursive function. This follows because $\Delta_{\langle a \rangle}^G(a_i) \geq n_i$ by definition, and by the construction from Lemma \ref{seq}, there is no recursive function satisfying this property. 
 
To complete the proof, it suffices to observe that the relative growth function is equivalent to $f(r) = \#\{n : l(n) \leq r\}$ is $o(r^2)$, by Lemma \ref{fpt}.

\end{proof}

\section{Extending Length Functions}\label{elf}

Let  $H \le G$, where $G$ is generated by a finite set $X$, and $H$ is generated by a finite set $Y$. Suppose that we have an arbitrary function $l : H \rightarrow \{0,1,2,\dots,\} $ satisfying the $(D)$ condition. In this section we will explore the question of when the function $l$ can be extended to $G$.

Here we introduce the funtion which will serve as the desired extension. The details are similar to those in the proof of Theorem \ref{2.2} Part $(2)$. Let $g \in G$. Then $g$ may be written as a product  $g= x_1 \cdots x_k$, where  $x_i$ is either an element of $X^{\pm 1}$ or  $x_i \in  H$.  For this factorization $P$ of $g$, we define
$L(g,P)= w(x_1)+\cdots+w(x_k)$, where $w(x_i)=1$  for the generators
from $X^{\pm 1}$  and  $w(x_i)=l (x_i)$  for  $x_i\in H.$
Finally, let \begin{equation}\label{bsvd}
L(g) =\min_P L(g,P) \textrm{ where $P$ varies over all such factorizations
of  $g$} 
\end{equation}

First we will show that $L$ defined in Equation \eqref{bsvd} is a length function. Let $g, h \in G$. Then for the specific factorizations $P$ of $g=x_1 \cdots x_k$ and $Q$ of $h=y_1 \cdots y_r$ satisfying $L(g)=L(g,P)$ and $L(h)=L(h,Q)$, we have the factorization $R$ of $gh$ given by $x_1 \cdots x_k y_1 \cdots y_r$. For this factorization $R$ we have $L(gh) \leq L(gh,R) = w(x_1)+\cdots+w(x_k)+w(y_1)+\cdots+w(y_r) = L(g,P)+L(h,Q) = L(g)+L(h)$.

Now, given a factorization $P$ of $g=x_1 \cdots x_k$ witnessing $L(g)=L(g,P)$ we have the factorization $Q$ of $g^{-1}$ given by $x_k^{-1} \cdots x_{1}^{-1}$. Then $L(g) = w(x_1)+\cdots+w(x_k) = w(x_1^{-1})+\cdots+w(x_k^{-1}) = L(g^{-1},Q) \geq L(g^{-1})$. By symmetry, $L(g)=L(g^{-1})$.

The following auxilliary definition will be a tool to prove that condition $(D3)$ is satisfied. Let $r \in \mathbb{N}$. We will compute $\# \{g \in G: L(g) \leq r\}$. 
\begin{defn}
Let $g \in G$ have $L(g) \leq n$. Let $L(g) = L(g,P)$ where $P$ is given by $g=x_1 \dots x_k$. We define the type $\tau(P)$ of the product $P$ as follows. One separates all the factors $x_1,\dots,x_k$ by commas, preserves all the factors $x_i \in X^{\pm 1}$, and replaces every $x_i\in H$ by the string $\$\dots\$ $ of length $l(x_i)$. 
\end{defn}
Thus $\tau(P)$ is a word over a finite alphabet (which includes the comma sign). Let $a$ be the number of letters in this finite alphabet. We will compute the number of choices for $g$ by multiplying the number of types $\tau(P)$ arising from a factorization $P$ with $L(g,P) \leq r$ by how many of each such type there are.

First, let us compute the number of types. Let $\tau(P)$ be any type arising from a factorization $P$ of $g$ with $L(g,P) \leq r$. The length of the type $\tau(P)$ is at most $2r$, and so the number of types of the products $P$ with $r$ factors is bounded by $a^{2r}$, a function equivalent to $2^r$. 

If the type is fixed, then to obtain a product $P$ of this type, one must replace the substrings  $\$\dots\$$ by elements of $H$. For a given string of length $n_i$, we have that there are at most $\# \{ h \in H: l(h) \leq n_i \} \leq b^{n_i}$ products of this type, for some $b$ arising from the hypothesis that $l$ is a length function. Since $\sum n_i\le r,$ we have at most $ b^{n_1}b^{n_2}\cdots b^{n_k}\le b^r$  products of the same type. Finally we note that the product of two exponential functions is an exponential function, and so condition $(D3)$ is satisfied for $L$. 

\begin{proof} of Theorem \ref{kam} Part $(1)$:
We will show that the length function $L$ is in fact an extension of $l$, where $l$ is the usual length function on $H$, $l$ is special, and $H \leq G$ is undistorted. Let $h \in H$. Then one factorization $P$ is $h=h$, and so $L(h,P)=w(h)=l(h)$. Therefore, $L(h) \leq l(h)$. 

To obtain the reverse inequality, we will use that the embedding is undistorted and that $l$ is special. Because the embedding is undistorted, there exists a positive integer $c$ so that for any $h \in H$, 
\begin{equation}\label{undistorted}
|h|_Y \leq c|h|_X.
\end{equation}

Let $h \in H$ and consider the patrition $P$ of $h$ so that $L(h,P)=L(h)$, where $P$ is $h=x_1\dots x_k,$ where for each $i$, either $x_i \in X^{\pm 1}$ or $x_i \in H$. Because $l$ is special, we have that there exists a function as in Definition \ref{special}, so that for $h \in H$, $l(h) = \Theta f(|h|_Y).$ That is, there exists a constant $e$ so that $l(h) \leq ef(|h|_Y)$.  By Equation \eqref{undistorted}, $f(|h|_Y)\le f(c|h|_X),$ because $f$ is nondecreasing. Because $f$ is subadditive, $f(c|h|_X) \le cf(|h|_X).$ Using the given partition and that $f$ is subadditive and nondecreasing, we have that $$cf(|h|_X) \le c f(|x_1|_X+\dots+|x_k|_X) \le c (f(|x_1|_X)+\dots+f(|x_k|_X)).$$ For each $i$ we have that $f(|x_i|_X) \le D w(x_i)$ for some positive constant $D$. Indeed, if $x_i \in X^{\pm 1}$ then $w(x_i)=|x_i|_X=1$ and so we choose a constant $d$ so that $f(1) \leq d$. If $x_i \in H$, then by Equation \eqref{undistorted}, and the properties of $f$, we have $$f(|x_i|_X)\le f(c |x_i|_Y)\le cf(|x_i|_Y).$$ Because $l$ is special, there is a constant $k$ so that $cf(|x_i|_Y) \leq kc l(x_i)$. Finally, since $x_i \in H$ and by definition of $w$ we have that $kc l(x_i) = kcw(x_i).$ Letting $D=\max\{d,kc\}$ we see that $f(|x_i|_X) \le D w(x_i)$. 
Hence $$l(h)\le ecD(w(x_1)+\dots + w(x_k))= CL(P,h)=CL(h)$$
where $C=ecD$ does not depend on $h$, as required.

\end{proof}

\begin{lemma}\label{oio}
Let $H$ be a finitely generated group and let $l : H \rightarrow \{0,1,\dots\} $ be a length function that is not special. There is an increasing sequence of integers $n_k$ and pairs $a_k, b_k$ of elements in $H$ such that $|a_k|_Y=n_k\ge 1$, $|b_k|_Y \geq n_k$ and $$l(a_k)\ge k^3 l(b_k), k \geq 1.$$ 
\end{lemma}

\begin{proof}
Consider two functions on the natural numbers defined as follows. Let $$f_1(r) = \max\{l(h) : |h|_Y =r\} \textrm{ and } f_2(r) = \min\{l(h) : |h|_Y \geq r\}.$$ If there exists a constant $C$ so that for all $r \in \mathbb{N}$ we have $\frac{f_1(r)}{f_2(r)} \leq C$ then for $h \in H$ we have $$f_2(|h|_Y) \leq l(h) \leq f_1(|h|_Y) \leq Cf_2(|h|_Y)$$ which implies that $f_2 \restriction_H = \Theta (l \restriction_H)$, contrary to our assumptions, since $f_2$ is also non-decreasing. Therefore, for any $k^3$ we have $n_k$ so that $f_1(n_k) > k^3 f_2(n_k)$. Thus for elements $a_k, b_k \in H$ with $|a_k|_Y=n_k, |b_k|_Y \geq n_k$ we have $l(a_k) \geq k^3 l(b_k)$ by definition of $f_1(n_k)=l(a_k)$ and $f_2(n_k)=l(b_k)$.
\end{proof}

\begin{proof} of Theorem \ref{kam} Part $(2)$:

As before, let $H$ be a finitely generated group, and let $l : H \rightarrow \{0,1,2,\dots,\} $ be any function satisfying the $(D)$ condition. Suppose that $l$ is not special.

We construct words in the alphabet $\{1,2\}$ $$v_k=\displaystyle\prod_{i=1}^k (1^i \cdot 2) =\epsilon_{k,1}\dots\epsilon_{k,m_k},$$ where $m_k = \frac{k(k+3)}{2}$ and each $\epsilon_{k,j} \in \{1,2\}$.

We define $G$ as the factor-group of the free product $H* \langle c \rangle$  over the normal
closure of the set $R=\{w_k\}_{k \geq 40}$ where $$w_k=a_kc^{\epsilon_{k,1}}b_kc^{\epsilon_{k,2}}b_k\dots c^{\epsilon_{k,m_k-1}}b_kc^{\epsilon_{k,m_k}}$$ and $a_k, b_k$ are the elements constructed in Lemma \ref{oio}.

Note that if $j-i\ge 2k-2$, then a subword $c^{\epsilon_{k,i}}b_kc^{\epsilon_{k,i+1}}b_k\dots c^{\epsilon_{k,j}}$ of the normal form $w_k$ is not equal in $H$ to a subword  $c^{\epsilon_{k',i'}}b_{k'}c^{\epsilon_{k',i'+1}}b_{k'}\dots c^{\epsilon_{k',j'}}$ unless $(k,i,j)= (k',i',j')$. This follows from the following observation. Any subword occurring twice in $v_k$ is a subword of $1^{k-2}21^{k-1}.$ In particular, every subword of length $2k-1$ is unique in $v_k$.

Hence the syllabic length of a piece ({\it cf.} \cite{ls} Chapter 5.9) is at most $2(2k-2)=4k-4$ while the syllabic length of $w_k$ is $2m_k=k(k+3).$ Since $k\ge 40,$ it follows that the presentation of $G$ as a quotient of the free product $H* \langle c \rangle$ satisfies the condition $C'(\frac{1}{10})$.

By the version of Greendlinger's lemma for free products (cf. \cite{ls} Ch. 5, Theorem 9.3) every
non-empty normal form of a word in $H* \langle c \rangle$ equal 1 in $G$ must contain a subword of a cyclic shift of some $w_k^{\pm 1}$ with syllabic length at least $(1-\frac{3}{10})k(k+3)$
(where $k(k+3)$ is the syllabic length of $w_k$). In particular such a normal
form must contain a non-trivial power of $c.$ Therefore $H$ is embedded in
$G$ by the canonical epimorphism $H* \langle c \rangle \to G.$

Next we will show that this embedding is undistorted. Indeed, if a geodesic
word $u$ in  $H$ is equal to a different geodesic word $u'$ in the generators of $H$ plus $c$,  then $u'$ must contain a subword $w'$, where $w'w''$ is a cyclic
shift of some $w_k^{\pm 1}$ and the syllabic length of $w'$ is greater then
$\frac{7}{10} k(k+3)$. It follows from the form of $w_k$ ($1\le\epsilon_{k,j}\le 2$, $|a_k|,|b_k|\ge k$) that the length of $w'$ in the generators of $H$ and $c$ is greater than
the length of $w''.$ Since $w'=w''^{-1}$ in $G$, $u'$ is not geodesic in $G$, a contradiction.

Consider the undistorted embedding of $H$ into the finitely generated group $G$ created above. It remains to show that $l$  does not extend to $G$.

Suppose by way of contradiction that there exists a function $L$ extending the function $l$ from $H$ to $G$ (up to $\Theta$ equivalence as in Definition \ref{theta}). Denote $L(c)=\alpha$.

 For the elements created in Lemma \ref{oio}, we have by hypothesis that the function is an extension that there exists a constant $\beta$ with $L(b_k)\le \beta l(b_k)$.  
Then for $k\ge 40$ we have
$$L(a_k)=  L(c^{\epsilon_{k,1}}b_kc^{\epsilon_{k,2}}b_k\dots c^{\epsilon_{k,m_{k}-1}}b_kc^{\epsilon_{k,m_k}})$$ by definition of $G$. Then $$L(c^{\epsilon_{k,1}}b_kc^{\epsilon_{k,2}}b_k\dots c^{\epsilon_{k,m_{k}-1}}b_kc^{\epsilon_{k,m_k}}) \leq \sum_{j=1}^{m_k} |\epsilon_{k,j}| \cdot L(c) +(m_{k}-1) L(b_k)$$ and by definition of $\epsilon_{k,j}, m_k$ and $L(c)$ as well as by Lemma \ref{oio} we have that $$\sum_{j=1}^{m_k} |\epsilon_{k,j}| \cdot L(c) + m_{k}L(b_k) \leq2m_k\alpha +m_{k}\beta l(b_k)$$ $$\leq k(k+3)(\alpha+\frac{\beta l(a_k)}{2k^3})\leq(\alpha+\beta)\frac{l(a_k)}{k}$$
Thus $L(a_k)=o(l(a_k))$, a contradiction.

\end{proof}

\begin{proof} of Theorem \ref{kam} Part $(3)$:
Let $H$ be distorted in $G$. Consider the function $l: H \rightarrow \mathbb{N}: l(h) = |h|_Y$. Then $l$ is special, and $l$ admits no extension to $G$. That is because if $L$ were such an extension, then $L=O(|g|_X)$ on $G$. Indeed, take a presentation of $g$ as a shortest product of generators $g=x_1\cdots x_k.$ Then $L(g) = L(x_1 \cdots x_k) \leq L(x_1) + \cdots + L(x_k) \leq ck =c|g|_X$ where $c=\max \{L(x):x \in X^{\pm1}\}$. If $l$ were equivalent to $L$ on $H$, then $H$ would be undistorted, a contradiction.
\end{proof}

\renewcommand{\abstractname}{Acknowledgements}
\begin{abstract}
The authors are grateful to Denis Osin for the valuable discussion and suggestions.
\end{abstract}

\bibliographystyle{abbrv} 

\begin{thebibliography}{MM} 

\bibitem[B]{bass} Bass, H., The Degree of Polynomial Growth of Finitely Generated Nilpotent Groups, Proc. London Math. Soc., Volume 3, Number 25, 603-614, 1972. 

\bibitem[Bau]{baumslag} Baumslag, G., Lecture Notes on Nilpotent Groups, Regional Conference Series in Mathematics, Number 2, Providence RI, 1971.

%\bibitem[D]{davis} Davis, T., Length Functions for Semigroup Embeddings, Algebra and Discrete Math, Volume 9, Number 1, 2010, 1-15.

%\bibitem[E]{partition} Erdos, P., On an elementary proof of some asymptotic formulas in the theory of partitions, Ann. Math., Volune 43, Number 2, 437–450, 1942.

\bibitem[F]{farb} Farb, B., The Extrinsic Geometry of Subgroups and the Generalized Word Problem, Proceedings of the London Math. Society, Volme 68, Number 3,  1994, 577-593.

%\bibitem[G]{grigorchuk} Grigorchuk, R., Symmetric Random Walks on Discrete Groups, Mnogokomponentnye sluchainye sistemy (Mulitcomponent Stochastic Systems), Dobrushin, R. and Sinai Ya., Eds., Moscow: Nauka, 132-152, 1978.

\bibitem[G]{grigorchuk2} Grigorchuk, R., On the Milnor Problem of Group Growth, Soviet Math. Dokl., Volume 28, Number 1, 23-26, 1983. Translated from Dokl. Akad. Nauk SSSR 271, 1, 30-33, 1983 (Russian).


\bibitem[GH]{gdlh} Grigorchuk R.I., De La Harpe P., On Problems Related to Growth, Entropy, and Spectrum in Group Theory, J. Dyn. and Contr. Syst., Volue 3, 51–89, 1997.

\bibitem[Gro]{gromov2} Gromov, M., Geometric Group Theory: Asymptotic Invariants of Infinite Groups, Volume 2 (Sussex, 1991), 1-295, London Mathematical Society Lecture Notes, Series 182, Cambridge University Press, 1993.

\bibitem[Gro2]{gromov} M. Gromov, Groups of Polynomial growth and Expanding Maps, Publications mathematiques I.H.E.S., Volume 53, 53-78, 1981

\bibitem[H]{hall} Hall, P., Finiteness Conditions for Solvable Groups, Proc. London Math. Soc., Volume 3, Number 4, 419-436, 1954

\bibitem[HaR]{partition} Hardy, G. H., Ramanujan, S., Asymptotic Formulae in Combinatory Analysis, Proc. London Math. Soc., Volume 17, 75-115, 1918.

\bibitem[LS]{ls} Lyndon, R., Schupp, P., Combinatorial Group Theory, Springer-Verlag, 274-278, 2001

\bibitem[O]{olshanskii} Olshanskii, A. Yu., Distortion Functions for Subgroups, Geometric Group Theory Down Under, Ed. Cossey, J, Miller, C. F. III, Neumann, W. D., Shapiro, M., Proceedings of a Special Year in Geometric Group Theory, Canberra, Australia, 1996, Walter de Gruyter, Berlin, New York, 281-291, 1999.

\bibitem[O2]{olshr} A.Yu.Olshanskii, On the subgroup distortion in finitely presented groups,
Matem. Sbornik, 188 (1997), Number 11, 73-120 (in Russian), English translation in
Mat. Sb., 188(1997), pp.51--98

\bibitem[OOs]{oo} Olshanskii, A., Osin, D., A Quasi-Isometric Embedding Theorem for Groups, arXiv:1202.6437.

\bibitem[Os]{osins} Osin, D., Exponential radicals of solvable Lie groups, J. Algebra, Volume 248 No. 2, 790–805, 2002

\bibitem[Os2]{osin} Osin, D., Problem of Intermediate Relative Growth of Subgroups in Solvable and Linear Groups, Proc. Stekov Inst. Math., Volume 231, 316-338, 2000

\bibitem[Os3]{osinnilp} Osin, D., Subgroup Distortions in Nilpotent Groups, Communications in Algebra, Volume 29, Number 12, 5439-5463, 2001

\bibitem[W]{wolf} Wolf, J., Growth of Finitely Generated Solvable Groups and Curvature of Riemannian Manifolds, Journal of Differential Geometry, Volume 2, 1968. 
 
 \end{thebibliography}
 \addtolength{\textwidth}{.7in}
\addtolength{\evensidemargin}{-0.35in}
\addtolength{\oddsidemargin}{-0.35in}
\addtolength{\textheight}{.5in}
\addtolength{\topmargin}{-.25in}

\end{document}